\documentclass[12pt]{amsart} %amsfonts,
\usepackage{amsmath,amssymb,amscd,amsthm,fullpage}
\usepackage{latexsym}
\usepackage{colortbl}

\newtheorem{theorem}{Theorem}%[section]

\newtheorem{lemma}[theorem]{Lemma}

\newtheorem{remark}[theorem]{Remark}
\def\e{\varepsilon}
\def\irr#1{{\rm Irr}(#1)}
\def\irrr#1#2 {\irr {#1 \mid #2}}
\newcommand{\R}{\mathbb R}
\newcommand{\F}{e^{-\frac{|y|^p}{p}}}

\begin{document}

\title{Maximal surface area of a convex set in $\R^n$ with respect to exponential rotation invariant measures.\\
\hfill\\
{\tiny{Journal of Mathematical Analysis and Applications, Volume 404, Issue 2, 15 August 2013, Pages 231-238}}}
 \subjclass[2010]{Primary:  44A12, 52A15, 52A21}
  \keywords{ convex bodies, convex polytopes, Surface area, Gaussian measures}
\author{Galyna Livshyts}
\address{Department of Mathematics, Kent State University,
Kent, OH 44242, USA} \email{glivshyt@kent.edu}
\date{March 2013}

\begin{abstract}
 Let $p$ be a positive number. Consider probability measure $\gamma_p$ with density
 $\varphi_p(y)=c_{n,p}e^{-\frac{|y|^p}{p}}$. We show that the maximal surface area of a convex body in $\R^n$ with respect to $\gamma_p$ is asymptotically equal to $C_p n^{\frac{3}{4}-\frac{1}{p}}$, where constant $C_p$ depends on $p$ only. This is a generalization of Ball's \cite{ball} and Nazarov's  \cite{fedia}
 bounds, which were given for the case of the standard Gaussian measure $\gamma_2$.
\end{abstract}
\maketitle
\section {Introduction}
As usual,  $|\cdot|$ denotes the norm in Euclidean $n$-space $\R^n$, and $|A|$ stands for the Lebesgue measure  of a measurable set $A \subset \R^n$. We will write     $B_2^n=\{x\in \R^n: |x| \le 1\}$ for the unit ball in $\R^n$, $S^{n-1}=\{x\in \R^n: |x|=1\}$ for the
 unit $n$-dimensional sphere. We will denote by  $\nu_n=|B_2^n|=\pi^{\frac{n}{2}}/\Gamma(\frac{n}{2}+1)$.

In this paper we will study the geometric properties of measures  $\gamma_p$ on $\R^n$ with density
$$\varphi_p(y)=c_{n,p}\F,$$
where $p\in (0,\infty)$ and $c_{n,p}$ is the normalizing constant.

Many interesting results are known for the  case $p=2$  (standard Gaussian measure). One must mention  the Gaussian  isoperimetric inequality of Borell \cite{B} and Sudakov, Tsirelson \cite{ST}:   fix some $a \in (0,1)$ and $\e >0$, then  among all measurable sets  $A \subset \R^n$, with $\gamma_2(A)=a$  the set for which $\gamma_2(A+\e B_2^n)$ has the smallest Gaussian measure is half-space. We refer to books \cite{bogachev} and \cite{LT} for more properties of Gaussian measure and inequalities of this type.

Mushtari and Kwapien asked the reverse version of isoperimetric inequality, i.e. how large the Gaussian surface area of a convex set $A\subset \R^n$ can be. In \cite{ball} it was shown, that Gaussian surface area of a convex body in $\R^n$ is asymptotically bounded by $C n^{\frac{1}{4}}$, where $C$ is an absolute constant. Nazarov in \cite{fedia} gave the complete solution to this problem by proving the sharpness of Ball's result:
$$0.28 n^{\frac{1}{4}}\leq \max \gamma_2(\partial Q)\leq 0.64 n^{\frac{1}{4}},$$
where maximum is taken over all convex bodies. Further estimates for $\gamma_2(\partial Q)$ were provided in \cite{kane}.

Isoperimetric inequalities for rotation invariant measures were studied by Sudakov, Tsirelson \cite{ST}, who proved that for a measure $\gamma$ with density $e^{-h(\log|x|)}$, where $h(t)$ is a positive convex function, there exist derivative of a function $M_Q(a)=\gamma(aQ)$ (where $Q$ is a convex body), and minimum of $M_Q'(1)$ among all convex bodies is attained on half spaces. Thus the result can be applied to measures $\gamma_p$ by setting   $h(t)=\frac{e^{pt}}{p}$. Some interesting results for manifolds with density were also provided by Bray and Morgan \cite{BM} and further generalized by Maurmann and Morgan \cite{MM}.

The main goal of this paper is to compliment the  study of isoperimetric problem for rotation invariant measures and to prove an inverse isoperimetric inequality for $\gamma_p$, which is done using the generalization of Nazarov's method from \cite{fedia}.

We remind that the surface area of a convex body $Q$ with respect to the measure $\gamma_p$ is defined to be
\begin{equation}\label{outer}
\gamma_p(\partial Q)=\liminf_{\epsilon\rightarrow +0}\frac{\gamma_p((Q+\epsilon B_2^n)\backslash Q)}{\epsilon}.
\end{equation}
One can also provide an integral formula for $\gamma_p(\partial Q)$:
\begin{equation}\label{def2}
\gamma_p(\partial Q)=\int_{\partial Q} \varphi_p(y)d\sigma(y)=c_{n,p}\int_{\partial Q} \F d\sigma(y),
\end{equation}
where $d\sigma(y)$ stands for Lebesgue surface measure. We refer to \cite{kane} for the proof in the case $p=2$.

%Note that doubled (\ref{quatient}) is a surface area with respect to $\gamma_p$ of a hyperplane passing through the origin.
The following theorem is the main result of this paper:
\begin{theorem}\label{th:main}
For any positive $p$
$$e^{-\frac{9}{4}} n^{\frac{3}{4}-\frac{1}{p}}\leq \max \gamma_p(\partial Q)\leq C(p) n^{\frac{3}{4}-\frac{1}{p}},$$
where $C(p)\approx 2\sqrt[4]{2\pi}c_1 e^{-(\frac{c_2}{p}+c_3 p)} p^{\frac{3}{4}}$.
\end{theorem}
In Theorem \ref{th:main} and further we will denote by  ''$\approx$''  an asymptotic equality while $p$   tends to infinity and by $c_1$, $c_2$, \dots different absolute constants. We shall also use notation $\precsim$ for an asymptotic inequality.

Using the trick from \cite{ball} one can find an easy estimate from above for the surface area by $e^{\frac{1}{p}-1}n^{1-\frac{1}{p}}.$ The calculation is given in the Section 2, as well as some other important preliminary facts. The upper bound from  Theorem \ref{th:main} is obtained in the Section 3, and the lower bound is shown in the Section 4.

\noindent {\bf Acknowledgment}. I would like to thank Artem Zvavitch and Fedor Nazarov for introducing me to the subject, suggesting me this problem and for extremely helpful and fruitful discussions.

\section{Preliminary lemmas.}

 We remind that $\gamma_p$ is a probability measure on $\R^n$ with density $\varphi_p(y)=c_{n,p}\F$, where $p\in (0,\infty)$. The normalizing constant $c_{n,p}$ equals to $[n\nu_n J_{n-1,p}]^{-1}$, where
\begin{equation}\label{maint}
J_{a,p}=\int_0^{\infty} t^{a} e^{-\frac{t^p}{p}}dt.
\end{equation}
%Taking to the attention that
%\begin{equation}\label{nu}
%\frac{\nu_n}{\nu_{n+1}}\approx \frac{1}{\sqrt{2\pi}}n^{\frac{1}{2}},
%\end{equation}
%we note, that
%\begin{equation}\label{quatient}
%\frac{c_{n+1,p}}{c_{n,p}}\approx \frac{1}{\sqrt{2\pi}}n^{\frac{1}{2}-\frac{1}{p}}.
%\end{equation}
We need to give an asymptotic estimate for $J_{a,p}$.  Our main tool is the Laplace method, which can be found, for example, in \cite{bruign}. For the sake of completeness, we shall present it here:

\begin{lemma}\label{laplace}
Let $h(x)$ be a function on an interval $(a,b)\ni 0$ having at least two continuous derivatives (here $a$ and $b$ may be infinities). Let $0$ be the global maxima point for $h(x)$ and assume for convinience that $h(0)=0$. Assume that for any $\delta>0$ there exist $\eta(\delta)>0$ s.t. for any $x\not\in [-\delta,\delta]$ $h(x)<-\eta(\delta)$. Assume also that $h''(0)<0$ and that the integral $\int_a^b e^{h(x)} dx<\infty$. Then
$$\int_a^b e^{th(x)}dx \approx \sqrt{-\frac{2\pi}{h''(0)t}},\,\,\,\,\,t\rightarrow \infty.$$
\end{lemma}
\begin{proof} First, using conditions of the lemma and Teylor formula, for a sufficiently small $h''(0)>>\epsilon>0$ there exist positive $\delta=\delta(\epsilon)$, such that for any $x\in (-\delta,\delta)$ it holds that \\$|h(x)-\frac{h''(0)x^2}{2}|\leq \frac{\epsilon x^2}{2}.$ Thus the integral
\begin{equation}\label{delta}
\int_{-\delta}^{\delta} e^{th(x)} dx\leq \frac{1}{\sqrt{-(h''(0)+\epsilon)}}\int_{-\delta\sqrt{-(h''(0)+\epsilon)}}^{\delta\sqrt{-(h''(0)+\epsilon)}} e^{\frac{ty^2}{2}} dy\leq \sqrt{-\frac{2\pi}{(h''(0)+\epsilon)t}}.
\end{equation}
Note that for any constant $C>0$,
\begin{equation}\label{hvost}
\int_C^{\infty} e^{\frac{-ty^2}{2}}dy\geq e^{\frac{-(t-1)C^2}{2}} \int_C^{\infty} e^{-\frac{y^2}{2}} dy=C' e^{-C'' t},
\end{equation}
thus (\ref{delta}) is asymptotically equivalent to $\sqrt{-\frac{2\pi}{(h''(0)+\epsilon)t}}$. It remains to prove that the whole integral is coming from the small interval about zero under the lemma conditions on $h(x)$. Indeed, for an arbitrary $\epsilon$ we choose $\delta(\epsilon)$, and then by condition of the lemma, we pick $\eta(\delta)=\eta(\epsilon)$, so that
$$\int_{(a, -\delta)\cup(\delta, b)} e^{th(x)}\leq e^{-(t-1)\eta(\delta)}\int_{a}^{b} e^{h(x)} dx=C'e^{-C''t}.$$
Thus,
$$\int_a^b e^{th(x)}dx \precsim \sqrt{-\frac{2\pi}{(h''(0)+\epsilon)t}}.$$
Similarly to (\ref{delta}) and by (\ref{hvost}), the reverse inequality holds:
$$\sqrt{-\frac{2\pi}{(h''(0)-\epsilon)t}} \precsim\int_a^b e^{th(x)}dx, ,\,\,\,\,\,t\rightarrow \infty.$$
Taking $\epsilon$ small enough we finish the proof.
\end{proof}
We will now apply the Laplace's method to deduce the asymptotic estimate for $J_{a,p}$.

\begin{lemma}\label{integral_s_a}
Let $p>0$. Then
$$J_{a,p} \approx \sqrt{\frac{2\pi}{p}} a^{\frac{1}{p}-\frac{1}{2}} a^{\frac{a}{p}} e^{-\frac{a}{p}},\,\,\mbox{ as }\,\, a\rightarrow \infty.$$
\end{lemma}
\begin{proof} We notice:
$$\int_0^{\infty} t^a e^{-\frac{t^p}{p}} dt =a^{\frac{a}{p}} e^{-\frac{a}{p}} \int_0^{\infty} e^{\frac{a}{p} \left(\log \frac{t^p}{a} - \frac{t^p}{a} +1\right)} dt=a^{\frac{a}{p}} e^{-\frac{a}{p}} a^{\frac{1}{p}} \int _0^{\infty} e^{\frac{a}{p} h(x)} dx,$$
where $h(x)=p\log x - x^p +1$.

Note that $h(1)=h'(1)=0$, and in addition $h''(1)=-p-p(p-1)=-p^2<0$. Also, $\int_{0}^{\infty} e^{h(x)}  dx=\int_0^{\infty} e ^{-c(p)x^p} dx<\infty.$ For any $\delta>0$ it holds that $h(x)<\eta(\delta)=-C(p) \delta^p$ outside of the interval $[-\delta,\delta]$. So one can apply Lemma \ref{laplace} to finish the proof.
\end{proof}
Next we shall observe that the surface area is mostly concentrated in a  narrow annulus. Define $\Delta_p=1-e^{-\frac{1}{p}}$. Note that $\Delta_p\in (0,1)$ while $p>0$. Let 
$$A_p=(1+\Delta_p)(n-1)^{\frac{1}{p}}B^n_2\setminus (1-\Delta_p)(n-1)^{\frac{1}{p}}B_2^n;$$ 
we shall call $A_p$ the concentration annulus.

\begin{lemma}\label{ball}
There exist  positive constants $C'(p)$ and $C''(p)$, depending on $p$ only, such that $\gamma_p(\partial Q \cap A_p^c)\leq C'(p)e^{-C''(p)n}$ for  any  convex body  $Q \subset \R^n$.
\end{lemma}
\begin{proof} First, assume that $|y|<(1-\Delta_p)(n-1)^{\frac{1}{p}}$ for any $y\in \partial Q'$. Then
\begin{equation}\label{smallpart}
\gamma_p(\partial Q')\leq \frac{1}{n\nu_n J_{n-1,p}}\int_{\partial Q'} e^{-\frac{|y|^p}{p}} d\sigma(y)\leq \frac{|\partial Q'|}{n\nu_n J_{n-1,p}}.
\end{equation}
Since $Q'\subset (1-\Delta_p)(n-1)^{\frac{1}{p}}B_2^n$, it holds that $|\partial Q'|\leq (1-\Delta_p)^{n-1}(n-1)^{\frac{n-1}{p}}n\nu_n$. By the choice of $\Delta_p$, (\ref{smallpart}) is exponentially small.

Assume now that for any $y\in \partial Q''$ it holds that $|y|>(1+\Delta_p)(n-1)^{\frac{1}{p}}$. We can rewrite the expression for $\gamma_p(\partial Q'')$ using a trick from \cite{ball}. Notice, that
$$e^{-\frac{|y|^p}{p}} = \int_{|y|}^{\infty} t^{p-1} e^{-\frac{t^p}{p}} dt=\int_0^{\infty} t^{p-1} e^{-\frac{t^p}{p}} \chi_{[-t,t]}(|y|) dt.$$
Under this assumptions on $y$, for any $t\leq (1+\Delta_p)(n-1)^{\frac{1}{p}}$ it holds that $\chi_{[-t,t]}(|y|)=0$ and
$$e^{-\frac{|y|^p}{p}} = \int_{(1+\Delta_p)(n-1)^{\frac{1}{p}}}^{\infty} t^{p-1} e^{-\frac{t^p}{p}} \chi_{[-t,t]}(|y|) dt.$$
Thus
\begin{eqnarray*}
\gamma_p(\partial Q'')&=&\frac{1}{n\nu_n J_{n-1,p}}\int_{\partial Q''} e^{-\frac{|y^p|}{p}} d\sigma(y)\\
&=&\frac{1}{n\nu_n J_{n-1,p}}\int_{\partial Q''} \int_{(1+\Delta_p)(n-1)^{\frac{1}{p}}}^{\infty} t^{p-1} e^{-\frac{t^p}{p}} \chi_{[-t,t]}(|y|) dt d\sigma(y)\\
& =& \frac{1}{n\nu_n J_{n-1,p}}\int_{(1+\Delta_p)(n-1)^{\frac{1}{p}}}^{\infty} t^{p-1} e^{-\frac{t^p}{p}} |\partial Q''\cap tB_2^n| dt\\
&\leq& \frac{1}{ J_{n-1,p}}\int_{(1+\Delta_p)(n-1)^{\frac{1}{p}}}^{\infty} t^{n+p-2} e^{-\frac{t^p}{p}} dt.
\end{eqnarray*}
From the previous lemmas it is clear that for any constant $\delta>0$, we get
$$\int_{(1+\delta)(n-1)^{\frac{1}{p}}}^{\infty} t^{n+p-2} e^{-\frac{t^p}{p}} dt\leq C'(p)e^{-C''(p)n},$$
for some positive $C'(p)$ and $C''(p)$. Thus
$$\gamma_p(\partial Q'')\leq \frac{C'(p)e^{-C''(p)n}}{n^{-\frac{1}{2}} n^{\frac{n}{p}} e^{-\frac{n}{p}}},$$
which is exponentially small as well.
\end{proof}

Note, that using same trick from \cite{ball}, one can obtain a rough bound for $\gamma_p$-surface area of a convex body. Namely,

$$\gamma_p(\partial Q)=\frac{1}{n\nu_n J_{n-1,p}}\int_{\partial Q} e^{-\frac{|x|^p}{p}}dx= \frac{1}{n\nu_n J_{n-1,p}}\int_0^{\infty}t^{p-1}e^{-\frac{t^p}{p}}|\partial Q\cap t B_2^n|dt\leq$$
$$\frac{J_{n+p-2,p}}{J_{n-1,p}}\approx n^{1-\frac{1}{p}},,\,\,\,\,\,n\rightarrow \infty.$$
This bound is not best possible. The next section is dedicated to the best possible asymptotic upper bound.

\section {Upper bound}

We will use the approach developed by Nazarov in \cite{fedia}. Let us consider ''polar'' coordinate system $x=X(y,t)$ in $\R^n$ with $y\in \partial Q$, $t>0$. Then
$$\int_{\R^n}\varphi_p(y)d\sigma(y)=\int_0^{\infty}\int_{\partial Q} D(y,t)\varphi_p(X(y,t)) d\sigma(y)dt,$$
where $D(y,t)$ is a Jacobian of $x\rightarrow X(y,t)$. Define
\begin{equation}\label{xi}
\xi(y)=\varphi_p^{-1}(y)\int_0^{\infty} D(y,t)\varphi_p(X(y,t)) dt.
\end{equation}
Then
$$1=\int_{\partial Q} \varphi_p(y)\xi(y)dy,$$
and thus
$$\int_{\partial{Q}}\varphi_p(y)dy\leq \frac{1}{\min\limits_{y\in\partial Q} \xi(y)}.$$
Following \cite{fedia}, we shall consider two such systems.

\subsection{First coordinate system}

Consider ''radial'' polar coordinate system $X_1(y,t)=yt$. The Jacobian $D_1(y,t)=t^{n-1}|y|\alpha$, where $\alpha=\alpha(y)$, denotes the absolute value of cosine of an angle
between $y$ and $\nu_y$. Here $\nu_y$ stands for a normal vector at $y$. From (\ref{xi}),
\begin{equation}\label{xi1}
\xi_1(y)=e^{\frac{|y|^p}{p}}\alpha |y|^{1-n}J_{n-1}\approx \sqrt{\frac{2\pi}{p}} e^{\frac{|y|^p}{p}}\alpha |y|^{1-n} n^{\frac{1}{p}-\frac{1}{2}} e^{F((n-1)^{\frac{1}{p}})},\,\,\,\,\,n\rightarrow\infty,
\end{equation}
where $F(t)=(n-1)\log t-\frac{t^p}{p}$. Since $(n-1)^{\frac{1}{p}}$ is the maxima point for $F(t)$, for all $y\in \R^n$, $F((n-1)^{\frac{1}{p}})\geq F(|y|)$. So we can estimate (\ref{xi1}) from below by

\begin{equation}\label{xi1final}
\xi_1(y)\gtrsim \sqrt{\frac{2\pi}{p}} n^{\frac{1}{p}-\frac{1}{2}}\alpha.
\end{equation}

\subsection{Second coordinate system}

Now consider ''normal'' polar coordinate system $X_2(y,t)=y+t\nu_y$. Then $D_2(y,t)\geq 1$ for all $y\not\in Q$. Thus, by cosine rule, namely, $|x+y|^2=x^2+y^2-2xy\cos \beta,$
where $\beta$ is an angle between vectors $x$ and $y$, we get:
\begin{equation}\label{startxi2}
\xi_2(y)\geq e^{\frac{|y|^p}{p}}\int_0^{\infty}e^{-\frac{(|y|^2+t^2+2t|y|\alpha)^{\frac{p}{2}}}{p}}dt.
\end{equation}
Note, that for any positive function $f(x)$ defined on the interval $I$,
\begin{equation}\label{chebychev1}
\int_{I} e^{-f(t)}dt\geq e^{-f(t_0)}|\{t:\,f(t)<f(t_0)\}\cap I|.
\end{equation}
Consider
$$f(t)=\frac{(|y|^2+t^2+2t|y|\alpha)^{\frac{p}{2}}}{p}.$$
By intermediate value theorem there is $t_1$ such that
\begin{equation}\label{t1}
(|y|^2+t_1^2+2t_1|y|\alpha)^{\frac{p}{2}}=|y|^p+1.
\end{equation}
Since $f(t)$ is increasing, from (\ref{chebychev1}) and (\ref{t1}) we get
$$\xi_2(y)\geq e^{-\frac{1}{p}}t_1.$$
Now we need to estimate $t_1$ from below. Using (\ref{t1}) and taking $y\in A_p$, we apply Mean Value Theorem and get
$$t_1=\sqrt{\alpha^2|y|^2-|y|^2+(|y|^p+1)^{\frac{2}{p}}}-\alpha|y|\approx \sqrt{\alpha^2|y|^2+\frac{2}{p}|y|^{2-p}}-\alpha|y|.$$
Multiplying the last expression by a conjugate and applying the inequality $\sqrt{a+b}\leq \sqrt{a}+\sqrt{b}$, we get:
\begin{equation}\label{xi2}
\xi_2(y)\geq e^{-\frac{1}{p}}\sqrt{\frac{2}{p}}|y|^{1-\frac{p}{2}}\frac{1}{1+\sqrt{2p}\alpha|y|^{\frac{p}{2}}}.
\end{equation}
Considering (\ref{xi1}) and (\ref{xi2}) with $|y|\in A_p$, we get
\begin{equation}\label{epic}
\xi(y):=\xi_1(y)+\xi_2(y)\gtrsim n^{\frac{1}{p}-\frac{1}{2}}\left(\sqrt{\frac{2\pi}{p}} \alpha+\frac{C}{C_1 \alpha \sqrt{n}+1}\right),
\end{equation}
where $C_1=\sqrt{2p}(2-e^{-\frac{1}{p}})^{\frac{p}{2}}$; for $0<p\leq 2$ $C=\sqrt{\frac{2}{p}}e^{\frac{1}{2}-\frac{2}{p}}$, and for $p\geq 2$ \\$C=\sqrt{\frac{2}{p}} e^{-\frac{1}{p}}(2-e^{-\frac{1}{p}})^{1-\frac{p}{2}}$.

Note that (\ref{epic}) is minimized whenever $\alpha=\sqrt[4]{\frac{p}{2\pi}}\sqrt{\frac{C}{C_1}}n^{-\frac{1}{4}}$. The minimal value of (\ref{epic}) is $C(p)^{-1} n^{\frac{1}{p}-\frac{3}{4}}$, where $C(p)=2\sqrt[4]{\frac{2\pi}{p}}\sqrt{\frac{C}{C_1}}$. This implies, that
$$\gamma_p(\partial Q\cap A_p)\leq C(p) n^{\frac{3}{4}-\frac{1}{p}}.$$
One can note that $C(p)$ tends to infinity while $p$ tends to infinity or to zero. Applying Lemma \ref{ball}, we finish the proof of the upper bound from the Theorem 1.

\begin{remark}
It was noticed by Nazarov, that his construction in \cite{fedia} also implies that any polytope $P_K$ with $K$ faces has Gaussian surface area bounded by $C\sqrt{\log K}$. The same way, in the general case $\gamma_p(\partial P_K)\leq C(p) n^{\frac{1}{2}-\frac{1}{p}} \sqrt{\log K}.$ Indeed, let $H(\rho)$ be a  hyperplane distanced at $\rho$ from the origin. By the Mean Value Theorem, the surface area of $H(\rho)$ is bounded from above by
\begin{equation}\label{hyp-rho}
\frac{1}{\sqrt{2\pi}} n^{\frac{1}{2}-\frac{1}{p}} e^{-\frac{\rho^2}{2}n^{1-\frac{2}{p}}}.
\end{equation}
By (\ref{xi2}), and since $\alpha |y|=\rho$ for $y\in H(\rho)$, we note that
\begin{equation}\label{polytope}
\gamma_p(\partial P_K)\lesssim \sum_{\rho\geq \sqrt{2\log K}n^{\frac{1}{p}-\frac{1}{2}}} \gamma_p (H(\rho)) + \left(e^{-\frac{1}{p}}\frac{2}{p}\frac{|y|^{2-p}}{\sqrt{\frac{2}{p}}|y|^{1-\frac{p}{2}}+2\sqrt{2\log K}n^{\frac{1}{p}-\frac{1}{2}}}\right)^{-1}.
\end{equation}

The first summand is about a constant times $n^{\frac{1}{2}-\frac{1}{p}}$ (by (\ref{hyp-rho})). The second summand is bounded by $C(p) n^{\frac{1}{2}-\frac{1}{p}} \sqrt{\log K}$.
\end{remark}

\section{Lower bound}
Let's consider N uniformly distributed random vectors $x_i\in S^{n-1}$. Let $\rho=n^{\frac{1}{p}-\frac{1}{4}}$ and $r=r_w=n^{\frac{1}{p}}+w$, where $w\in [-W,W]$, and $W=n^{\frac{1}{p}-\frac{1}{2}}$. Consider random polytope $Q$ in $\R^{n}$, defined as follows:
$$Q=\{x\in \R^{n} :\, <x,x_i>\leq \rho, \,\,\,\forall i=1,...,N\}.$$
The expectation of $\gamma_p(\partial Q)$ is
\begin{equation}\label{expectation}
\frac{1}{n\nu_n J_{n-1}} N\int_{\R^{n-1}} \exp(-\frac{(|y|^2+\rho^2)^{\frac{p}{2}}}{p})(1-p(|y|))^{N-1}dy,
\end{equation}
where $p(t)$ is the probability that the fixed point on the sphere of radius $\sqrt{t^2+\rho^2}$ is separated from the origin by hyperplane $<x,x_i>=\rho$.

Passing to polar coordinates, we shall estimate (\ref{expectation}) from below by
\begin{equation}\label{expectation1}
\frac{\nu_{n-1}}{\nu_n J_{n-1,p}} N\int_{W}^{W} f(n^{\frac{1}{p}}+w)(1-p(r_w))^{N-1}dy,
\end{equation}
where $f(t)=t^{n-2} e^{-\frac{(t^2+\rho^2)^{\frac{p}{2}}}{p}}$. Note, that $\frac{\nu_{n-1}}{\nu_n}\approx \frac{\sqrt{n}}{\sqrt{2\pi}}$. Thus we estimate (\ref{expectation1}) from below by
\begin{equation}\label{expectation2}
\frac{1}{\sqrt{2\pi}} n n^{-\frac{n}{p}} e^{\frac{n}{p}} f(n^{\frac{1}{p}}+W) N\int_{W}^{W} (1-p(r_w))^{N-1}dy.
\end{equation}
Next,

$$f(n^{\frac{1}{p}}+W)\geq n^{\frac{n-2}{p}}(1+n^{-\frac{1}{2}})^{n-2} e^{-\frac{n}{p}} e^{-\frac{3}{2}\sqrt{n}}\approx n^{\frac{n}{p}} e^{-\frac{n}{p}} n^{-\frac{2}{p}} e^{-\frac{\sqrt{n}}{2}}.$$
Thus (\ref{expectation2}) is greater than

\begin{equation}\label{expectation3}
\frac{1}{\sqrt{2\pi}} n^{1-\frac{2}{p}} e^{-\frac{\sqrt{n}}{2}} N\int_{W}^{W} (1-p(r_w))^{N-1}dy.
\end{equation}
Next, we estimate the probability $p(r)$. The same way, as in \cite{fedia}, by Fubbini Theorem,
\begin{equation}\label{uzhas}
p(r)=(\int_{-\sqrt{r^2+\rho^2}}^{\sqrt{r^2+\rho^2}}(1-\frac{t^2}{r^2+\rho^2})^{\frac{n-3}{2}}dt)^{-1}
\int_{\rho}^{\sqrt{r^2+\rho^2}}(1-\frac{t^2}{r^2+\rho^2})^{\frac{n-3}{2}}dt.
\end{equation}
Directly by Laplace method (or due to the fact that it represents the sphere surface area) the first integral is approximately equal to $\sqrt{2\pi}n^{\frac{1}{p}-\frac{1}{2}}.$

Using an elementary inequality that $1-a\leq e^{-\frac{a^2}{2}}e^{-a}$, for all $a>0$, one can estimate the second integral in (\ref{uzhas}) by
$$\int_{\rho}^{\infty}exp(-\frac{n-3}{4(r^2+\rho^2)^2}t^{4})\cdot exp(-\frac{n-3}{r^2+\rho^2}\frac{t^2}{2})dt$$
$$\leq exp(-\frac{n-3}{4(r^2+\rho^2)^2}\rho^{4})\int_{\rho}^{\infty} exp(-\frac{n-3}{r^2+\rho^2}\frac{t^2}{2})dt.$$
The first multiple is of order $e^{-\frac{1}{4}}$ under these assumptions on $r$ and $\rho$. The second integral can be estimated with
usage of inequality
$$\int_{\rho}^{\infty}e^{-a\frac{t^2}{2}}\leq \frac{1}{a\rho}e^{-a\frac{\rho^2}{2}}.$$
We note that $a\rho^2$ is of order $\frac{n-2}{\rho^2+r^2}\sim n^{\frac{1}{2}}(1-3n^{-\frac{1}{2}})$ up to an additive error $\sim n^{-\frac{1}{2}}$. Hence one can write that
\begin{equation}\label{probability}
p(r)\leq \frac{e^{\frac{5}{4}}}{\sqrt{2\pi}} n^{-\frac{1}{4}} e^{-\frac{\sqrt{n}}{2}}.
\end{equation}
Now, one can choose $N=\frac{\sqrt{2\pi}}{e^{\frac{5}{4}}} n^{\frac{1}{4}} e^{\frac{\sqrt{n}}{2}}$. From (\ref{expectation3}) and (\ref{probability}) it now follows that the expectation of a $\gamma_p$-surface area is greater than
%$$\frac{1}{\sqrt{2\pi}} n^{1-\frac{2}{p}} e^{-\frac{\sqrt{n}}{2}} \frac{\sqrt{2\pi}}{e^{\frac{1}{4}+C}} n^{\frac{1}{4}} e^{\frac{\sqrt{n}}{2}} 2 C n^{\frac{1}{p}-\frac{1}{2}} e^{-1}=$$
$$e^{-\frac{1}{4}} n^{\frac{3}{4}-\frac{1}{p}},$$
which finishes the proof of the Theorem 1. $\square$

\end{document}